\documentclass[a4, 12pt]{amsart}
\usepackage{amssymb}
\usepackage{amstext}
\usepackage{amsmath}
\usepackage{amscd}
\usepackage{latexsym}
\usepackage{amsfonts}
\usepackage{color}
\usepackage{enumerate}
\usepackage{textcomp}
\usepackage[all]{xy}
\usepackage[dvips]{graphicx}

\theoremstyle{plain}
\newtheorem{thm}{Theorem}[section]
\newtheorem*{thm*}{Theorem}
\newtheorem*{cor*}{Corollary}
\newtheorem*{defn*}{Definition}

\newtheorem{lem}[thm]{Lemma}
\newtheorem{cor}[thm]{Corollary}

\newtheorem*{claim*}{Claim}

\theoremstyle{definition}

\newtheorem{ex}[thm]{Example}
\newtheorem{rem}[thm]{Remark}

\theoremstyle{remark}

\numberwithin{equation}{thm}

\def\Ext{\mathrm{Ext}}

\def\e{\mathrm{e}}
\def\m{\mathfrak m}
\def\n{\mathfrak n}

\def\p{\mathfrak p}

\def\Z{\Bbb Z}

\def\H{\mathrm{H}}

\newcommand{\rmc}{\mathrm{c}}

\newcommand{\rme}{\mathrm{e}}

\newcommand{\rmr}{\mathrm{r}}

\newcommand{\rmH}{\mathrm{H}}

\newcommand{\rmQ}{\mathrm{Q}}

\newcommand{\fkm}{\mathfrak{m}}
\newcommand{\fkn}{\mathfrak{n}}

\newcommand{\fkp}{\mathfrak{p}}

\begin{document}

\setlength{\baselineskip}{14pt}

\title{Bounds for the first Hilbert coefficients of $\m$-primary ideals}
\pagestyle{plain}
\author{Asuki Koura}
\address{Otsuma Nakano Junior and Senior High School, 2-3-7, Kamitakata, Nakano-ku, Tokyo 164-0002, Japan}
\email{koura@otsumanakano.ac.jp}

\author{Naoki Taniguchi}
\address{Department of Mathematics, Graduate Courses, School of Science and Technology, Meiji University, 1-1-1 Higashi-mita, Tama-ku, Kawasaki 214-8571, Japan}
\email{taniguti@math.meiji.ac.jp}

\thanks{{AMS 2010 {\em Mathematics Subject Classification:}
13D40, 13H10, 13H15.}}
\begin{abstract}
This paper purposes to characterize Noetherian local rings $(A,\fkm)$ of positive dimension such that the first Hilbert coefficients of $\fkm$-primary ideals in $A$ range among only finitely many values. Examples are explored.
\end{abstract}

\maketitle



\section{Introduction} 
Let $A$ be a commutative Noetherian local ring with maximal ideal $\m$ and $d = \dim{A} > 0$. For each $\m$-primary ideal $I$ in $A$ we set
$$
\H_I(n) = \ell_A(A/I^{n+1}) \ \  \mbox{for} \ \  n \ge 0
$$
and call it the Hilbert function of $A$ with respect to $I$, where $\ell_A(A/I^{n+1})$ denotes the length of the $A$-module $A/I^{n+1}$. Then there exist integers $\{\e_i(I)\}_{0 \le i \le d}$ such that 
$$
\H_I(n) = \rme_0(I) \binom{n+d}{d} - \rme_1(I) \binom{n+d-1}{d-1} + \cdots + (-1)^d \rme_d(I) \ \ \mbox{for all}\ \ n \gg 0.$$ The integers $\e_i(I)'s$ are called the  Hilbert coefficients of $A$ with respect to $I$. These integers describe the complexity of given local rings, and there are a huge number of  preceding papers  about them, e.g., \cite{GhGHOPV1, GhGHOPV2, GMP, GNO1, GNO2}. In particular, the integer $\e_0(I) > 0$ is called the multiplicity of $A$ with respect to $I$ and has been explored very intensively.  One of the most spectacular results on the multiplicity theory  says that \textit{$A$ is a regular local ring if and only if $\e_0({\m}) = 1$}, provided $A$ is unmixed. This was proven by P. Samuel \cite{Samuel} in the case where $A$ contains a field of characteristic $0$ and then by  M. Nagata \cite{Nagata2} in the above form. Recall that a local ring $A$ is $unmixed$, if $\dim {\widehat{A}} = \dim {\widehat{A}/\p}$ for every associated prime ideal $\p$ of the $\fkm$-adic completion $\widehat{A}$ of $A$. The Cohen-Macaulayness in $A$ is characterized in terms of $\e_0(Q)$ of parameter ideals $Q$ of $A$. On the other hand, in \cite{GhGHOPV1} the authors analyzed the boundness of the values $\e_1(Q)$ for parameter ideals $Q$ of $A$ and deduced that the local cohomology modules $\{\rmH^i_\fkm(A)\}_{i \ne d}$ are finitely generated, once $A$ is unmixed and the set $\Lambda (A) = \{\e_1(Q) \mid Q \  \mbox{is a parameter ideal of} \ A\}$ is finite.

In the present paper we focus on the first Hilbert coefficients $\e_1(I)$ for $\fkm$-primary ideals $I$ of $A$. Our study dates back to  the paper of M. Narita \cite{Na}, who showed that if $A$ is a Cohen-Macaulay local ring, then  $\e_1(I) \ge 0$, and also $\e_2(I) \ge 0$ when $d = \dim A  \ge 2$. We consider the set 
$$
\Delta (A) = \{\e_1(I) \mid I\text{ is an }\m\text{-primary ideal in }A \}
$$ and are interested in the problem of when $\Delta (A)$ is finite. Under the light of  Narita's theorem, if $A$ is a Cohen-Macaulay local ring of positive dimension, our problem is equivalent to the question of  when the values $\e_1(I)$ has a finite upper bound, and Theorem \ref{1.1} below settles the question, showing that such Cohen-Macaulay local rings are exactly of dimension one and analytically unramified, where $\rmH_\fkm^0(A)$ denotes the $0$-th local cohomology module of $A$ with respect to $\fkm$.

\begin{thm}\label{1.1}
Let $(A, \m)$ be a Noetherian local ring with $d=\dim A > 0$. Then the following conditions are equivalent.
\begin{enumerate}[$(1)$]
\item $\Delta(A)$ is a finite set.
\item $d=1$ and $A/\rmH_\fkm^0(A)$ is analytically unramified.
\end{enumerate}
\end{thm}

We prove Theorem \ref{1.1} in Section 3. Section 2 is devoted to preliminaries for the proof.  Let $\overline{A}$ denote the integral closure of $A$ in the total ring of fractions of $A$. The key  is the following, which we shall prove in Section $2$.

\begin{thm}\label{1.2}
Let $(A, \m)$ be a Cohen-Macaulay local ring with $\dim A=1$. Then
$$
\sup \Delta(A) = \ell_{A}(\overline{A}/A).
$$
Hence $\Delta(A)$ is a finite set if and only if $A$ is analytically unramified.
\end{thm}

\noindent
For the proof we need particular calculation of $\e_1(I)$ in one-dimensional Cohen-Macaulay local rings, which we  explain also in Section 2.

When $A = k[[t^{a_1},t^{a_2},\ldots, t^{a_\ell}]]$ is the semigroup ring of a numerical semigroup $H 
=\{\sum_{i=1}^\ell c_ia_i \mid 0 \le c_ i \in \Bbb Z\}$ over a field $k$ (here $t$ is the indeterminate over $k$ and $0 < a_1 < a_2 < \ldots <a_{\ell}$ be integers such that $\operatorname{gcd}(a_1, a_2, \ldots, a_\ell) = 1$), the set $\Delta (A)$ is finite and $\Delta (A) = \{0, 1, \ldots, \sharp (\Bbb N \setminus H)\}$, where $\Bbb N$ denotes the set of non-negative integers (Example \ref{4.1}). However, despite  this result and the fact $\sup \Delta(A) = \ell_{A}(\overline{A}/A)$ in Theorem \ref{1.2}, the equality
$$
\Delta(A) = \{ n \in \Bbb Z \mid 0 \le n \le \ell_A(\overline{A}/A) \}
$$
does not necessarily hold true in general.
In Section 4 we will explore several concrete examples, including  an example for which the equality is not true (Example \ref{4.7}).

Unless otherwise specified, throughout this paper let $A$ be a Noetherian local ring with maximal ideal $\fkm$ and $d = \dim A > 0$. Let $\rmQ (A)$ denote the total ring of fractions of $A$. For each finitely generated $A$-module $M$, let $\ell_A(M)$ and $\mu_A(M)$ denote respectively the length of and the number of elements in a minimal system of generators of $M$.


\section{Proof of Theorem \ref{1.2}}
In this section let $(A, \m)$ be a Cohen-Macaulay local ring with $\dim A = 1$. Let $I$ be an $\m$-primary ideal of $A$ and assume that $I$ contains a parameter ideal $Q = (a)$ as a reduction. Hence there exists an integer $r \ge 0$ such that $I^{r+1} = QI^r$. This assumption is automatically satisfied, when the residue class field $A/\m$ of $A$ is infinite. We set
$$
\frac{I^n}{a^n} =\left\{ \frac{x}{a^n} \left| \frac{}{} \right. x \in I^n \right\} \subseteq \rmQ (A) \ \  \mbox{for} \ \  n \ge 0
$$
and let
$$
B = A \biggl[ \frac{I}{a} \biggr] \subseteq {\rmQ (A)},
$$
where $\rmQ (A)$ denotes the total ring of fractions of $A$. Then
$$
B ={\bigcup_{n \ge 0}\frac{I^n}{a^n}}= \frac{I^r}{a^r} \cong I^r
$$
as an $A$-module, because ${\frac{I^n}{a^n} = \frac{I^r}{a^r}}$ if $n \geq r$ as $\frac{I^n}{a^n}\subseteq \frac{I^{n+1}}{a^{n+1}}$ for all $n \ge 0$. Therefore $B$ is a finitely generated $A$-module, whence $A\subseteq B \subseteq \overline{A}$, where $\overline{A}$ denotes the integral closure of $A$ in $\rmQ (A)$. We furthermore have the following.

\begin{lem}[{\cite[Lemma 2.1]{GMP}}]\label{2.1} 
\item[$(1)$] $\e_0(I) = \ell_{A}(A/Q)$.
\item[$(2)$] $\e_1(I) = \ell_{A}(I^r/Q^r)=\ell_{A}(B/A) \leq \ell_{A}(\overline{A}/A)$.
\end{lem}

Conversely, let $A \subseteq B \subseteq \overline{A}$ be an arbitrary intermediate ring and assume that $B$ is a finitely generated $A$-algebra.  We choose a  non-zerodivisor $a \in \m$ of $A$ so that $aB \subsetneq A$ and set $I = aB$.　   
Then $I$ is an $\fkm$-primary ideal of $A$ and 
$
I^2 = a^2B  = aI.
$
Hence $B=A\displaystyle{\left[ \frac{I}{a} \right] = \frac{I}{a}}$, so that we get the following.

\begin{cor}\label{2.2}$
\ell_A(B/A) = \e_1(I) \in \Delta(A).
$
\end{cor}

Let us  note the following.

\begin{lem}\label{2.3}
Let $(A, \m)$ be a Cohen-Macaulay local ring with $\dim A = 1$. Then
$$
\sup \Delta(A) \ge \ell_A(\overline{A}/A).
$$
\end{lem}

\begin{proof}
We set $s = \sup \Delta(A)$. Assume $s < \ell_A(\overline{A}/A)$ and choose elements $y_1,y_2,\ldots,y_{\ell}$ of $\overline{A}$ so that $\ell_A(\left[\sum_{i=1}^\ell Ay_i\right]/A) > s$. We consider the ring $B = A[y_1,y_2,\ldots,y_{\ell}]$. Then $A \subseteq B \subseteq \overline{A}$ and 
$$
s < \ell_A(\left[\sum_{i=1}^\ell Ay_i\right]/A) \le \ell_A(B/A),
$$
which is impossible, as $\ell_A(B/A) \in \Delta(A)$ by Corollary \ref{2.2}. Hence $s \ge \ell_A(\overline{A}/A)$.
\end{proof}

The assumption in the following corollary \ref{2.4} that the field $A/\fkm$ is infinite is necessary to assure a given $\fkm$-primary ideal $I$ of $A$ the existence of a reduction generated by a single element. We notice that even if the field $A/\fkm$ is finite, the existence is guaranteed when $\overline{A}$ is a discrete valuation ring (see Section 4).

\begin{cor}\label{2.4}
Let $(A, \m)$ be a Cohen-Macaulay local ring with $\dim A = 1$. Suppose that the field $A/\m$ is infinite. We then have
\[
\resizebox{\hsize}{!}{$
\Delta(A) = \{ \ell_A(B/A) \mid A \subseteq B \subseteq \overline{A}~\text{ is an intermediate ring which is a module-finite extension of}~A\}
$}.\]
\end{cor}

\begin{proof} Let $\Gamma (A)$ denote the set of the right hand side.
Let $I$ be an $\m$-primary ideal of $A$ and choose a reduction $Q=(a)$ of $I$. We put $B=A\displaystyle{\left[ \frac{I}{a} \right]}$. Then  $B$ is a module-finite extension of $A$ and Lemma \ref{2.1} (2) shows $\e_1(I) = \ell_A(B/A)$. Hence $\Delta(A) \subseteq  \Gamma (A)$. The reverse inclusion follows from Corollary \ref{2.2}.
\end{proof}

We finish the proof of Theorem \ref{1.2}.

\begin{proof}[Proof of Theorem $\ref{1.2}$]
By Lemma \ref{2.3} it suffices to show $\sup \Delta(A) \le \ell_A(\overline{A}/A)$. Enlarging the residue class field $A/\fkm$ of $A$, we may assume that the field $A/\fkm$ is infinite. Let $I$ be an $\m$-primary ideal of $A$ and choose $a \in I$ so that $aA$ is a reduction of $I$. Then 
$$
\e_1(I) \leq \ell_{A}(\overline{A}/A)
$$
by Lemma \ref{2.1} (2). Hence the result.
\end{proof}


\section{Proof of Theorem \ref{1.1}}
Let us prove Theorem \ref{1.1}. Let $(A, \m)$ be a Noetherian local ring with $d = \dim A > 0$. We begin with the following.

\begin{lem}\label{3.1}
Suppose that $\Delta(A)$ is a finite set. Then $d = 1$.
\end{lem}

\begin{proof} Let $I$ be an $\m$-primary ideal of $A$. Then for all  $k \ge 1$
$$
\rme_0(I^k) = k^d{\cdot}\rme_0(I) \ \ \mbox{and}\ \  \rme_1(I^k) = \frac{d-1}{2}{\cdot}\rme_0(I){\cdot}k^d + \frac{2\rme_1(I) - \rme_0(I){\cdot}(d-1)}{2}{\cdot}k^{d-1}.
$$
In fact, we have
$$
(1) \ \ \ 
\ell_A(A/({I^k})^{n+1}) = \e_0(I^k)\binom{n+d}{d}-\e_1(I^k)\binom{n+d-1}{d-1}+\cdots+(-1)^d\e_d(I^k)
$$
for $n\gg 0$, while 
\begin{eqnarray*}
(2) \ \ \ \ell_A(A/({I^k})^{n+1}) &=& \ell_A(A/I^{(kn+k-1)+1})\\
                        &=& \e_0(I)\binom{(kn+k-1)+d}{d}-\e_1(I)\binom{(kn+k-1)+d-1}{d-1}\\
                        & &  + \cdots+(-1)^d\e_d(I), 
\end{eqnarray*}
\begin{eqnarray*}
\binom{kn+k+d-1}{d} &=&k^d\binom{n+d}{d}+ a \binom{n+d - 1}{d-1} + (lower~ terms), \ \ \mbox{and} \\
\binom{kn+k+d-2}{d-1} &=&k^{d-1}\binom{n+d-1}{d-1} + (lower~ terms), 
\end{eqnarray*}
where
$$
a= k^{d-1}\Big(k+\frac{d-1}{2}\Big) - \frac{k^d}{2}(d+1).
$$
Comparing the coefficients of $n^d$ in equations $(1)$ and $(2)$, we see
$$
\rme_0(I^k) = k^d{\cdot}\rme_0(I).
$$
We similarly have
\begin{eqnarray*}
\e_1(I^k) &=& -\e_0(I)a +\e_1(I)k^{d-1}\\
　　　　　&=& -\e_0(I)\Big(k^d+\frac{d-1}{2}k^{d-1} - \frac{d+1}{2}k^d\Big) +\e_1(I)k^{d-1} \ \ \mbox{and}\ \ \\
          &=& \frac{d-1}{2}{\cdot}\rme_0(I){\cdot}k^d + \frac{2\rme_1(I) - \rme_0(I){\cdot}(d-1)}{2}{\cdot}k^{d-1},
\end{eqnarray*}
considering $n^{d-1}$.
Hence $d=1$, if the set $\{\e_1(I^k) \mid k \ge 1\}$ is finite.
\end{proof}

Lemma \ref{3.1} and the following estimations finish the proof of Theorem \ref{1.1}.  Remember that $\overline{A}$ is a finitely generated $A$-module if and only if the $\fkm$-adic completion $\widehat{A}$ of $A$ is a reduced ring, provided $A$ is a Cohen-Macaulay local ring with $\dim A = 1$.

\begin{thm}\label{3.2}
Let $(A,\m)$ be a Noetherian local ring with $\dim A = 1$ and set $B= A/\H_\fkm^0(A)$. Then
\begin{eqnarray*}
\sup\Delta(A) &=& \ell_{B}(\overline{B}/B)-\ell_A(\rmH_\fkm^0(A)) \ \ \ \text{and} \\
\inf\Delta(A) &=& -\ell_A(\rmH_\fkm^0(A)).
\end{eqnarray*}
\end{thm}

\begin{proof}
We set $W = \H_\fkm^0(A)$. Then $B = A/W$ is a Cohen-Macaulay local ring with $\dim A = 1$. Let $I$ be an $\m$-primary ideal of $A$.  We consider the exact sequence
$$
0 \rightarrow W/[I^{n+1} \cap W] \rightarrow A/I^{n+1} \rightarrow B/I^{n+1}B \rightarrow 0
$$
of $A$-modules. Then since $I^{n+1} \cap W=(0)$ for all $n \gg 0$, 
\begin{eqnarray*}
  \ell_A(A/I^{n+1}) &=& \ell_A(B/I^{n+1}B) + \ell_A(W) \\
                    &=& \e_0(IB)\binom{n+1}{1} - \e_1(IB) + \ell_A(W).
\end{eqnarray*}
Hence
$$
\e_0(I) = \e_0(IB)\ \ \text{and}\ \ 
\e_1(I) = \e_1(IB) - \ell_A(W) \ge -\ell_A(W),
$$
because $\e_1(IB) \ge 0$ by Lemma \ref{2.1} (2). If $I$ is a parameter ideal of $A$, then $IB$ is a parameter ideal of $B$ and
$$
\e_1(I) = \e_1(IB) - \ell_A(W) = -\ell_A(W).
$$
Thus from Theorem \ref{1.2} the estimations
\begin{eqnarray*}
\sup \Delta(A) & = & \sup \Delta(B)-\ell_A(W)   \ \ \\
                & = & \ell_B(\overline{B}/B) - \ell_A(W) \ \ \text{and}     \\
\inf \Delta(A) & = & - \ell_A(W)
\end{eqnarray*}
follow,  since every $\m B$-primary ideal $J$ of $B$ has the form $J = IB$ for some $\fkm$-primary ideal $I$ of $A$.
\end{proof}


\section{Examples}

We explore concrete examples. Let $0 < a_1 < a_2 < \ldots < a_\ell~(\ell \ge 1)$ be integers such that  $\operatorname{gcd}(a_1, a_2, \ldots, a_\ell)=1$. Let $V =k[[t]]$ be the formal power series ring over a field $k$. We set $A = k[[t^{a_1},t^{a_2},  \ldots, t^{a_\ell}]]$ and $H = \left<\sum_{i=1}^\ell c_ia_i \mid 0 \le c_i \in \Bbb Z \right>$. Hence $A$ is the semigroup ring of the numerical semigroup $H$. We have $V=\overline{A}$ and $\ell_A(V/A) = \sharp (\Bbb N \setminus H)$, where $\Bbb N = \{n \mid 0 \le n \in \Bbb Z\}$. Let $c=\rmc(H)$ be the conductor of $H$.

\begin{ex}\label{4.1}
Let $q = \sharp (\Bbb N \setminus H)$. Then $
\Delta(A)  = \{0, 1, \ldots, q\}.
$
\end{ex}

\begin{proof}
We may assume $q \ge 1$, whence $c \ge 2$. We write $\Bbb N \setminus H=\{c_1, c_2, \ldots, c_q \}$ with $1=c_1<c_2< \cdots <c_q=c-1 $ and set $B_i=A[t^{c_i},t^{c_{i+1}},\ldots ,t^{c_q}]$ for  each $1 \le i \le q$. Then the descending chain
$
V =B_1 \supsetneq B_2 \supsetneq \cdots \supsetneq B_q \supsetneq B_{q+1}:=A
$
of $A$-algebras gives rise to a composition series of the $A$-module $V/A$, since $\ell_A(V/A) = q$. Therefore $\ell_A(B_i/A)=q+1-i$ for all $1 \leq i \leq q+1$ and hence, setting $a=t^c$ and $I_i=aB_i ~(\subsetneq A)$, by Corollary \ref{2.2} we have $\e_1(I_i)=q+1-i$. Thus $\Delta (A) = \{ 0,1,\ldots ,q \}$ as asserted.
\end{proof}

Because $q = \frac{\rmc (H)}{2}$ if $H$ is symmetric (that is $A= k[[t^{a_1},t^{a_2},  \ldots, t^{a_\ell}]]$ is a Gorenstein ring), we readily have the following.

\begin{cor}\label{4.2}
Suppose that $H$ is symmetric. Then $\Delta (A) = \{ 0,1,\ldots ,\frac{\rmc(H)}{2} \}$.
\end{cor}

\begin{cor}\label{4.3}
Let $A = k[[t^{a},t^{a+1},  \ldots, t^{2a-1}]]$~$(a \ge 2)$. Then $\Delta (A) = \{0, 1, \ldots, a-1\}$. For the ideal 
$
I = (t^a, t^{a+1}, \ldots, t^{2a-2})
$ of $A$, one has
$$
\rme_1(I)= \left\{
\begin{array}{ll}
\rmr (A) -1 & (a = 2),\\
\rmr (A) & (a \ge 3)
\end{array}
\right.
$$ where $\rmr (A) =\ell_A(\Ext_A^1(A/\fkm, A))$ denotes the Cohen-Macaulay type of $A$.
\end{cor}

\begin{proof}
See Example \ref{4.1} for the first assertion. Let us check the second one. If $a=2$, then $A$ is a Gorenstein ring and $I$ is a parameter ideal of $A$, so that  $\e_1(I) = \rmr(A) -1 ~(=0)$. Let $a \ge 3$ and put $Q =(t^a)$. Then $Q$ is a reduction of $I$, since $IV = QV$. Because $A\displaystyle{\left[ \frac{I}{t^a} \right]}=k[[t]]$ and $\m = t^aV$, we get $A:_{\rmQ(A)}\m = k[[t]]$. Thus $
\e_1(I) = \ell_A(k[[t]]/A) = \ell_A([A:_{\rmQ (A)}\m]/A)=\rmr(A)
$
(\cite[Bemerkung 1.21]{HK}).
\end{proof}

\begin{rem}\label{4.4}
In Example \ref{4.3} $I$ is a canonical ideal of $A$ (\cite{HK}). Therefore  the equality  $\e_1(I) = \rmr(A)$ shows that  if $a \ge 3$, $A$ is not a Gorenstein ring but an almost Gorenstein  ring in the sense of \cite[Corollary 3.12]{GMP}.
\end{rem}

Let us consider local rings which are not analytically irreducible.

\begin{ex}\label{4.5}
Let $(R,\n)$ be a regular local ring with $n = \dim R \ge 2$. Let $X_1,X_2, \ldots, X_n$ be a regular system of parameters of $S$ and set $P_i =(X_j \mid 1 \le j \le n, ~j \ne i)$ for each $1 \le i \le n$. We consider the ring
$
A = R 
/ \bigcap_{i=1}^n P_i.
$
Then $A$ is a one-dimensional Cohen-Macaulay local ring with $\Delta (A) = \{0, 1, \ldots, n -1\}$.
\end{ex}

\begin{proof}
Let $x_i$ denote the  image of $X_i$ in $A$. We put $\fkp_i =(x_j \mid 1 \le j \le n,~j \ne i)$ and $\displaystyle B = \prod_{i=1}^n(A/\fkp_i)$. Then the homomorphism
$
\varphi : A \to B, ~a \mapsto (\overline{a}, \overline{a}, \ldots, \overline{a})$
is injective and $B = \overline{A}$. Since $\m B = \m$ and $\mu_A(B) = n$, $\ell_A(B/A) = n-1$. Let $\mathbf{e}_j =(0, \ldots, 0, \overset{j}{\check{1}}, 0, \ldots, 0)$ for  $1 \le j \le n$ and $\mathbf{e} = \displaystyle{\sum_{j=1}^n \mathbf{e}_j}$. Then $B=A\mathbf{e} +  \sum_{j=1}^{n-1}A\mathbf{e}_j$. We set $
B_i = A\mathbf{e} +  \sum_{j=1}^{i}A\mathbf{e}_j
$
for each $1 \le i \le n-1$. Then since $B_i =A[\mathbf{e}_1,\mathbf{e}_2,\ldots,\mathbf{e}_i]$, $B_i$ is a finitely generated $A$-algebra and $B_i \subsetneq B_{i+1}$. Hence 
$B=B_{n-1} \supsetneq B_{n-2} \supsetneq \cdots \supsetneq B_1 \supsetneq B_0:=A
$
gives rise to a composition series of the $A$-module $B/A$. Hence  $\Delta (A) = \{0, 1,\ldots, n -1\}$, as $\ell_A(B_i/A) = i$ for all $0 \le i \le n-1$.
\end{proof}

Let $A$ be a one-dimensional Cohen-Macaulay local ring. If $A$ is not a reduced ring, then the set $\Delta(A)$ must be infinite. Let us note one concrete example.

\begin{ex}
Let $V$ be a discrete valuation ring and let $A=V \ltimes V$ denote the  idealization of $V$ over $V$ itself. Then $\Delta (A) = \{n \mid 0 \le n \in \Z \}$.
\end{ex}

\begin{proof}
Let $K = \rmQ(V)$.  Then $\rmQ(A) = K \ltimes K$ and $\overline{A} = V \ltimes K$. We set $B_n = V \ltimes \displaystyle{\left (V{\cdot} \frac{1}{t^n} \right)}$ for $n \ge 0$. Then $A \subseteq B_n \subseteq \overline{A}$ and
\begin{eqnarray*}
\ell_A(B_n/A) &=& \ell_V(B_n/A) \\
              &=& \ell_V([V \oplus \displaystyle{\left (V{\cdot} \frac{1}{t^n} \right)}] /[V \oplus V]) \\
              &=& \ell_V(V{\cdot} \frac{1}{t^n}/V) \\
              &=& \ell_V(V/t^nV) \\
              &=& n.
\end{eqnarray*}
Hence $n \in \Delta(A)$ by  Corollary \ref{2.2}, so that $
\Delta (A) = \{n \mid 0 \le n \in \Z \}.
$
\end{proof}

\begin{ex}\label{4.7}
Let $K/k~(K \ne k)$ be a finite extension of fields and assume that there is no proper intermediate field between $K$ and $k$. Let $n=[K:k]$ and choose a $k$-basis $\{\omega_i\}_{1 \le i \le n}$ of $K$. Let $K[[t]]$ be the formal power series ring over $K$ and set $A = k[[\omega_1t, \omega_2t, \ldots, \omega_nt]]~\subseteq K[[t]]$. Then $\Delta (A) = \{0, n-1\}$.
\end{ex}

\begin{proof}
Let $V= k[[t]]$. Then $V= \displaystyle{\sum_{i=1}^n A \omega_i}$ and $V = \overline{A}$. Since $tV \subseteq A$, we have $\n=tV = \m$,  where $\m$ and $\n$ stands for the maximal ideals of $A$ and $V$, respectively. Therefore $\ell_A(V/A)=n-1$. Let $A \subseteq B \subseteq V$ be an intermediate ring. Then $B$ is a local ring. Let $\fkm_B$ denote the maximal ideal of $B$. We then have $\m= \m_B =\n$ since $\fkm = \fkn$ and therefore, considering the extension of residue class fields $
k = A/\fkn \subseteq B/\fkn \subseteq K = V/\fkn,$
we get $V=B$ or $B=A$. Since $V = \overline{A}$ is a discrete valuation ring, every $\m$-primary ideal of $A$ contains a reduction generated by a single element. Hence $\Delta (A) = \{ 0, n-1 \}$ by Corollary \ref{2.4}.
\end{proof}

\bigskip


\begin{thebibliography}{GKT}



\bibitem{GhGHOPV1} {L. Ghezzi, S. Goto, J. Hong, K. Ozeki, T. T.
Phuong and W. V. Vasconcelos, {\it Cohen--Macaulayness versus the vanishing of the first Hilbert coefficient of parameter ideals}, J. London Math. Soc. {\bf 81} (2010) 679--695.}

\bibitem{GhGHOPV2} {L. Ghezzi, S. Goto, J. Hong, K. Ozeki, T. T. Phuong and W. V. Vasconcelos, {\it The Chern and Euler coefficients of modules}, Preprint (2010).}



\bibitem{GMP} {S. Goto, N. Matsuoka and T. T. Phuong, {\it Almost Gorenstein rings}, J. Algebra, {\bf 379} (2013), 355-381.}

\bibitem{GNO1} S. Goto, K. Nishida, and K. Ozeki, {\it Sally modules of rank one}, Michigan Math. J., {\bf 57} (2008), 359-381.


\bibitem{GNO2}{ S. Goto, K. Nishida, and K. Ozeki, {\it The structure of Sally modules of rank one}, Math. Res. Lett., {\bf 15} (2008) 881--892.}




\bibitem{HK}{J. Herzog and E. Kunz, {\it Der kanonische Modul eines. Cohen-Macaulay-Rings}, Lecture Notes in Mathematics, {\bf 238}, Springer-Verlag (1971)}







\bibitem{Nagata2} M. Nagata, {\it The theory of multiplicity in general local rings}, Proceedings of the international symposium on algebraic number theory, Tokyo--Nikko, 1955, 191--226. Science Council of Japan, Tokyo, 1956.






\bibitem{Na} {M. Narita, {\it A note on the coefficients of Hilbert characteristic functions in semi-regular local rings}, Proc. Cambridge Philos. Soc., {\bf  59} (1963),  269--275.}


\bibitem{Samuel} P. Samuel, {\it La notion de multiplicit{\'e} en alg{\`e}bre et en g{\'e}om{\'e}trie alg{\'e}brique} (French), J. Math. Pures Appl., {\bf 30} (1951), 159--205. 
\end{thebibliography}
\end{document}